\documentclass[12pt,reqno]{amsart}
\usepackage{amsmath, amsfonts, amssymb, amsthm, hyperref,cite,color}
\usepackage{bm}
\usepackage{booktabs}
\usepackage{array}
\usepackage{mathrsfs}
\usepackage{enumitem}
\allowdisplaybreaks[4]
\def\l{\left}
\def\r{\right}
\def\bg{\bigg}
\def\({\bg(}
\def\){\bg)}
\def\t{\text}
\def\f{\frac}

\def\ord{{\rm ord}}
\def\ind{{\rm ind}}

\def\eq{\equiv}

\def\Z{\mathbb Z}

\def\N{\mathbb N}

\def\p{\mathfrak p}
\def\<{\langle}
\def\>{\rangle}
\def\1{{\bf 1}}

\theoremstyle{plain}
\newtheorem{theorem}{Theorem}[section]

\newtheorem{lemma}{Lemma}[section]
\newtheorem{corollary}{Corollary}[section]
\theoremstyle{definition}

\newtheorem*{Acks}{Acknowledgments}
\theoremstyle{remark}

\numberwithin{equation}{section}

\begin{document}

\title[A supercongruence related to Whipple's ${}_5F_4$ formula]{A supercongruence related to Whipple's ${}_5F_4$ formula and Dwork's dash operation}

\author[Chen Wang]{Chen Wang}
\address{(Chen Wang) Department of Applied Mathematics, Nanjing Forestry University, Nanjing 210037, People's Republic of China}
\email{cwang@smail.nju.edu.cn}

\author[He-Xia Ni]{He-Xia Ni*}
\address{(He-Xia Ni) Department of Applied Mathematics, Nanjing Audit University, Nanjing 211815, People's Republic of China}
\email{nihexia@yeah.net}

\begin{abstract}
We establish a parametric supercongruence related to Whipple's ${}_5F_4$ formula and Dwork's dash operation. As a typical consequence, we obtain the following result: for any prime $p\equiv3\pmod4$ and odd integer $r\geq1$,
$$
\sum_{k=0}^{p^r-1}(8k+1)\frac{(\frac14)_k^3(\frac12)_k}{(1)_k^3(\frac34)_k}\equiv 3p^r+\frac{27p^{3r}}{4}H_{(p^r-3)/4}^{(2)}\pmod{p^{r+3}},
$$
where $(x)_n=x(x+1)\cdots(x+n-1)$ is the Pochhammer symbol and $H_n^{(2)}=\sum_{k=1}^n\f{1}{k^2}$ is the $n$-th harmonic number of order $2$. This confirms a conjecture of Guo and Zhao [Forum Math. 38 (2026), 1099-1109]. Our proof rely on a new parametric WZ pair which allows us to transform the original sum to a computable form in the sense of congruence. Another essential ingredient of our proof involves the properties of Dwork's dash operation.
\end{abstract}

\thanks{$^{\ast}$Corresponding author}

\maketitle

\section{Introduction}

In the 1910s, Ramanujan announced some convergent hypergeometric series related to $1/\pi$ without proofs (cf. \cite{Berndt1994}), such as
\begin{equation}\label{RamaSer}
\sum_{k=0}^{\infty}(8k+1)\f{(\f14)_k^4}{(1)_k^4}=\f{2\sqrt2}{\sqrt{\pi}\Gamma(\f34)^2},
\end{equation}
where $(x)_n=x(x+1)\cdots(x+n-1)$ is the Pochhammer symbol and $\Gamma(x)$ is the classical Gamma function. \eqref{RamaSer} was finally confirmed by Hardy \cite{Hardy1923}.

In 1997, Van Hamme \cite{VanHamme1997} observed that the truncated forms of original Ramanujan-type series possess good congruence properties. For example, corresponding to \eqref{RamaSer}, Van Hamme conjectured the following supercongruence: for any prime $p\eq1\pmod4$,
\begin{equation}\label{VanHammeconj}
\sum_{k=0}^{(p-1)/4}(8k+1)\f{(\f14)_k^4}{(1)_k^4}\eq p\f{\Gamma_p(\f12)\Gamma_p(\f14)}{\Gamma_p(\f34)}\pmod{p^3},
\end{equation}
where $\Gamma_p(x)$ denotes the $p$-adic Gamma function introduced by Morita \cite{Morita1975}. Nowadays, we usually refer to \eqref{VanHammeconj} as a $p$-adic analogue of \eqref{RamaSer}. All of Van Hamme's observations have now been confirmed by different authors using various techniques (see, e.g., \cite{Long2011,McCarthy-Osburn2008,Mortenson2008,Osburn-Zudilin2016,Swisher2015}). In particular, Swisher \cite{Swisher2015} proved that \eqref{VanHammeconj} holds modulo $p^4$ and established the following associated supercongruence: for any prime $p\eq3\pmod4$,
\begin{equation}\label{Swisherextension}
\sum_{k=0}^{(3p-1)/4}(8k+1)\f{(\f14)_k^4}{(1)_k^4}\eq -\f{3}{2}p^2\f{\Gamma_p(\f12)\Gamma_p(\f14)}{\Gamma_p(\f34)}\pmod{p^4}.
\end{equation}
Note that $(1/4)_k\eq0\pmod{p}$ for $(p+3)/4\leq k\leq p-1$ if $p\eq1\pmod4$ or $(3p+3)/4\leq k\leq p-1$ if $p\eq3\pmod4$. Therefore, both the upper limits of sums on the left-hand side of \eqref{VanHammeconj} and \eqref{Swisherextension} can be replaced with $p-1$. In 2022, Pan, Tauraso and Wang \cite{Pan-Tauraso-Wang2022} established some parametric extensions of \eqref{VanHammeconj} and \eqref{Swisherextension}. For instance, they showed that for any odd prime $p$ and $p$-adic integer $\alpha$ with $\<-\alpha\>_p\geq (p+1)/2$,
\begin{equation}\label{Pan-Tauraso-Wangres}
\sum_{k=0}^{p-1}\f{2k+\alpha}{\alpha}\cdot\f{(\alpha)_k^4}{(1)_k^4}\eq p^2\alpha^{\ast}(2\alpha^{\ast}-1)\f{\Gamma_p(1-2\alpha)}{\Gamma_p(1+\alpha)\Gamma_p(1-\alpha)^3}\pmod{p^4},
\end{equation}
where for any $p$-adic integer $x$, $\<x\>_{p^r}$ stands for the least nonnegative residue of $x$ modulo $p^r$ and $x^{\ast}=(x+\<-x\>_p)/p$ denotes Dwork's dash operation (cf. \cite{Dwork1969}). Clearly, \eqref{Swisherextension} is the special case $\alpha=1/4$ of  \eqref{Pan-Tauraso-Wangres}.
For other parametric extensions of \eqref{VanHammeconj}, we refer the reader to \cite{Barman-Saikia2020, Guo2025, Guo-Schlosser2020, Liu-Wang2021, Liu-Wang2022, Pan-Tauraso-Wang2022}.

Recently, using the creative microscoping method (cf. \cite{Guo-Zudilin2019}), Guo and Zhao \cite{Guo-Zhao2026} studied some  $q$-supercongruences from a very-well-poised ${}_6\phi_5$ basic hypergeometric identity. As consequences, they obtained the following results: for any prime $p\eq1\pmod4$ and integer $r\geq1$, 
\begin{equation}\label{GuoZhaores1}
\sum_{k=0}^{(p^r-1)/2}(8k+1)\f{(\f14)_k^3(\f12)_k}{(1)_k^3(\f34)_k}\eq p^r\pmod{p^{r+3}},
\end{equation}
and for any prime $p\eq3\pmod4$ and odd integer $r\geq1$,
\begin{equation}\label{GuoZhaores2}
\sum_{k=0}^{p^r-1}(8k+1)\f{(\f14)_k^3(\f12)_k}{(1)_k^3(\f34)_k}\eq 3p^r\pmod{p^{r+2}}.
\end{equation}
Guo and Zhao \cite[Conjectures 7.1 and 7.2]{Guo-Zhao2026} also conjectured that \eqref{GuoZhaores1} still holds modulo $p^{r+5}$ for $p>5$ and \eqref{GuoZhaores2} can be extended to the modulus $p^{r+3}$ case as follows:
\begin{equation}\label{GuoZhaoconj7.2}
\sum_{k=0}^{p^r-1}(8k+1)\f{(\f14)_k^3(\f12)_k}{(1)_k^3(\f34)_k}\eq 3p^r+\f{27}{4}p^{3r}\sum_{j=1}^{(p^r-3)/4}\f{1}{j^2}\pmod{p^{r+3}}.
\end{equation}
This is our initial motivation. 

Recall Whipple's ${}_5F_4$ formula (cf. \cite{Whipple1926})
\begin{align}\label{Whipple}
&\sum_{k=0}^{\infty}\f{(a)_k(1+\f{a}{2})_k(b)_k(c)_k(d)_k}{(1)_k(\f a2)_k(1+a-b)_k(1+a-c)_k(1+a-d)_k}\notag\\
&\qquad=\f{\Gamma(1+a-b)\Gamma(1+a-c)\Gamma(1+a-d)\Gamma(1+a-b-c-d)}{\Gamma(1+a)\Gamma(1+a-b-c)\Gamma(1+a-b-d)\Gamma(1+a-c-d)}.
\end{align}
Clearly, \eqref{RamaSer} is the special case $a=b=c=d=1/4$ of \eqref{Whipple}. Meanwhile, \eqref{Pan-Tauraso-Wangres} and \eqref{GuoZhaores2} are $p$-adic analogues of \eqref{Whipple} in the case $a=b=c=d=\alpha$ and the case $a=b=c=1/4,\ d=1/2$, respectively. Motivated by \eqref{Pan-Tauraso-Wangres}, it is natural to ask whether \eqref{GuoZhaores2} or \eqref{GuoZhaoconj7.2} has a parametric extension. This is the second motivation.

Before stating our main result, we first introduce some notations. For $n\in\Z^+=\{1,2,3,\ldots\}$ and $x\in\Z$, use $\<x\>_n$ to denote the least nonnegative residue of $x$ modulo $n$. Let $p$ be an odd prime. Similarly as in $\Z$, for any $p$-adic integer $x$ and $r\in\Z^+$, $\<x\>_{p^r}$ stands for the least nonnegative residue of $x$ modulo $p^r$ and $x^{\ast}=(x+\<-x\>_p)/p$ denotes Dwork's dash operation on $x$ (cf. \cite{Dwork1969}). For convenience, for $n\in\Z^+$, use $x^{\ast_n}$ to represent iterating the dash operation on $x$ $n$ times, that is,
$$
x^{\ast_1}=x^{\ast},\quad x^{\ast_n}=(x^{\ast_{n-1}})^{\ast}\quad (n=2,3,4,\ldots).
$$
In particular, set $x^{\ast_0}=x$. For $n\in\N=\{0,1,2,\ldots\},\ m\in\Z^+$, the $n$-th harmonic number of order $m$ is defined by
$$
H_n^{(m)}:=\sum_{k=1}^n\f{1}{k^m}.
$$

Our main purpose is to establish a $p$-adic analogue of Whipple's formula \eqref{Whipple} with $a=b=c=\alpha$ and $d=\f12$.

\begin{theorem}\label{mainth1}
Let $c,d,s\in\Z^+$ with $d\geq 2$, $1\leq c,s\leq d$ and $\gcd(cs,d)=1$, and let $p\geq5$ be a prime with $p\eq s\pmod d$. Then, for any $r\in\Z^+$ with $(\f12+\alpha)^{\ast_r}(\f12+\alpha^{\ast_r})\not\eq0\pmod{p}$,  we have
\begin{equation}\label{mainth1eq}
\sum_{k=0}^{p^r-1}(2k+\alpha)\f{(\alpha)_k^3(\f12)_k}{(1)_k^3(\f12+\alpha)_k}\eq \alpha^{\ast_r}p^r -\f{(\alpha^{\ast_r})^3}{(\f12+\alpha)^{\ast_r}}p^{r+2}H_{\alpha^{\ast_r}p-\alpha^{\ast_{r-1}}}^{(2)}\pmod{p^{r+3}},
\end{equation}
where $\alpha=c/d$.
\end{theorem}

Theorem \ref{mainth1} seems quite strange. When $\<-1/2-\alpha\>_p<\min\{\<-\alpha\>_p,(p-1)/2\}$, the summands on the left-hand side of \eqref{mainth1eq} are not always $p$-adic integers. However, the sum of these summands is a $p$-adic integer. This phenomenon renders it difficult for us to prove Theorem \ref{mainth1} directly using the formula \eqref{Whipple} and conventional congruence techniques. To overcome this obstacle, we find a new parametric WZ pair (cf. \cite{PWZ}) which allows us to transform the original sum to a computable form in the sense of congruence. This idea is crucial in our proof.

In particular, putting $d=4,s=3,c=1$ and requiring $r$ to be odd in Theorem \ref{mainth1}, we have the following result.
\begin{corollary}\label{cor}
Guo and Zhao's conjectural supercongruence \eqref{GuoZhaoconj7.2} \cite[Conjecture 7.2]{Guo-Zhao2026} is true.
\end{corollary}
Moreover, Theorem \ref{mainth1} with $d=2,s=1,c=1$ gives that
\begin{equation}\label{cor1}
\sum_{k=0}^{p^r-1}(4k+1)\f{(\f12)_k^4}{(1)_k^4}\eq p^r \pmod{p^{r+3}},
\end{equation}
where we have used \eqref{2ordharmonic}. Note that the $r=1$ case of \eqref{cor1} is a stronger version of Van Hamme's (C.2) supercongruence \cite{VanHamme1997} and was first proved by Long \cite{Long2011}. 

For the convenience of interested readers, in Table \ref{examples}, we provide some concrete examples of the parameters in Theorem \ref{mainth1} for future use.

\begin{table}[htbp]
    \centering
    \caption{Examples of the parameters in Theorem \ref{mainth1}}
    \label{examples}
    \renewcommand{\arraystretch}{1.5}
    \begin{tabular}{>{\centering\arraybackslash}p{2cm}
                      >{\centering\arraybackslash}p{2cm}
                      >{\centering\arraybackslash}p{2cm}
                      >{\centering\arraybackslash}p{3cm}
                      >{\centering\arraybackslash}p{3cm}}
        \toprule 
        $d$ & $s$ & $\alpha$ & $\alpha^{\ast_r}$ & $(1/2+\alpha)^{\ast_r}$\\
        \midrule 
        $2$ & $1$ & $1/2$  & $1/2$ & $1$\\
        $3$ & $1$ & $1/3$  & $1/3$ & $5/6$\\
        $3$ & $1$ & $2/3$  & $2/3$ & $1/6$\\
        $3$ & $1$ & $1/6$  & $1/6$ & $2/3$\\
        $3$ & $1$ & $5/6$  & $5/6$ & $1/3$\\
        $3$ & $2$ & $1/3$  & $(3-(-1)^r)/6$ & $(3+2(-1)^r)/6$\\
        $3$ & $2$ & $2/3$  & $(3+(-1)^r)/6$ & $(3-2(-1)^r)/6$\\
        $3$ & $2$ & $1/6$  & $(3-2(-1)^r)/6$ & $(3+(-1)^r)/6$\\
        $3$ & $2$ & $5/6$  & $(3+2(-1)^r)/6$ & $(3-(-1)^r)/6$\\
        $4$ & $1$ & $1/4$ & $1/4$ & $3/4$\\
        $4$ & $1$ & $3/4$ & $3/4$ & $1/4$\\
        $4$ & $3$ & $1/4$ & $(2-(-1)^r)/4$ & $(2+(-1)^r)/4$\\
        $4$ & $3$ & $3/4$ & $(2+(-1)^r)/4$ & $(2-(-1)^r)/4$\\
        \bottomrule
    \end{tabular}
\end{table}

We briefly outline this paper. In Section \ref{sec2}, we prove some properties of Dwork's dash operation and give some immediate applications which play essential roles in the subsequent proof. We shall prove Theorem \ref{mainth1} and Corollary \ref{cor} in Section \ref{sec3}.

\section{Properties of Dwork's dash operation and immediate applications}\label{sec2}
For convenience, in the subsequent proof we shall write $\<-\alpha\>_{p^r}=a$. In Lemmas \ref{dashres} and \ref{u}, we determine the exact values of $\alpha^{\ast_r}\ (r\in\Z^+)$. 

\begin{lemma}\label{dashres}
Let $c,d,s\in\Z^+$ with $d\geq2$, $1\leq c,s\leq d$ and $\gcd(cs,d)=1$. Then, for any prime $p\eq s\pmod{d}$, we have
$$
\alpha^{\ast}=\f{\<s^{-1}c\>_d}{d},
$$
where $s^{-1}$ stands for the inverse of $s$ modulo $d$.
\end{lemma}

\begin{proof}
Clearly, $\<s^{-1}c\>_dp-c\eq 0\pmod{d}$ and $(\<s^{-1}c\>_dp-c)/d\eq -c/d\pmod{p}$. Meanwhile,
$$
-1<-\f{c}{d}<\f{p-c}{d}\leq\f{\<s^{-1}c\>_dp-c}{d}\leq \f{(d-1)p-c}{d}<p.
$$
Therefore,
$$
\f{\<s^{-1}c\>_dp-c}{d}\in\{0,1,2,\ldots,p-1\},
$$
which means
$$
\l\<-\f{c}{d}\r\>_p=\f{\<s^{-1}c\>_dp-c}{d}.
$$
Then we have
$$
\l(\f{c}{d}\r)^{\ast}=\f{\f{c}{d}+\<-\f{c}{d}\>_p}{p}=\f{\f{c}{d}+\f{\<s^{-1}c\>_dp-c}{d}}{p}=\f{\<s^{-1}c\>_d}{d}
$$
as desired.
\end{proof}

By Lemma \ref{dashres}, we have the following immediate corollaries.

\begin{corollary}\label{dashrep}
Under the conditions of Lemma \ref{dashres},
$$
\alpha^{\ast_r}=\f{\<s^{-r}c\>_d}{d}.
$$
\end{corollary}

\begin{corollary}\label{dashper}
Under the conditions of Lemma \ref{dashres}, the least positive period of the dash operation on $c/d$ is $\ind_ds$, where $\ind_ds$ denotes the index of $s$ modulo $d$, i.e., the least positive integer $n$ such that $s^n\eq1\pmod{d}$.
\end{corollary}

For instance, for $s=1$, we have $\ind_ds=1$, which implies $\alpha^{\ast_r}=\alpha$ for any $r\in\Z^+$. For $s=-1$, we have $\ind_ds=2$, which implies that
$$
\alpha^{\ast_r}=\begin{cases}\alpha,\quad&\t{if}\ 2\mid r,\\ \alpha^{\ast},\quad&\t{if}\ 2\nmid r.\end{cases}
$$

\begin{lemma}\label{u}
Under the conditions of Lemma \ref{dashres}, we have
$$
\alpha^{\ast_r}=\f{\<s^{-r}c\>_d}{d}=\f{\alpha+a}{p^r}.
$$
\end{lemma}

\begin{proof}
Clearly, $p^r\eq s^r\pmod{d}$. Therefore, $\<s^{-r}c\>_dp^r-c\eq0\pmod{d}$ and $(\<s^{-r}c\>_dp^r-c)/d\eq-c/d\pmod{p^r}$. Moreover,
$$
-1<-\f{c}{d}<\f{p^r-c}{d}\leq\f{\<s^{-r}c\>_dp^r-c}{d}\leq\f{(d-1)p^r-c}{d}<p^r,
$$
and then
$$
\f{\<s^{-r}c\>_dp^r-c}{d}\in\{0,1,2,\ldots,p^r-1\},
$$
which means
$$
\l\<-\f{c}{d}\r\>_{p^r}=\f{\<s^{-r}c\>_dp^r-c}{d}.
$$
So we have
$$
\f{\f{c}{d}+\l\<-\f{c}{d}\r\>_{p^r}}{p^r}=\f{\f{c}{d}+\f{\<s^{-r}c\>_dp^r-c}{d}}{p^r}=\f{\<s^{-r}c\>_d}{d}.
$$
In view of Corollary \ref{dashrep}, we concludes the proof.
\end{proof}

The next lemma concerns the largest multiple of $p$ in a certain set defined via the dash operation. 

\begin{lemma}\label{dashreduce}
Under the conditions of Theorem \ref{mainth1} with $r>1$, for $j\in\{0,1,2,\ldots,r-2\}$, we have
\begin{equation}\label{dashreduceeq}
\max\{l\in\{1,2,3,\ldots,\alpha^{\ast_r}p^{r-j}-\alpha^{\ast_j}\}:p\mid l\}=\alpha^{\ast_r}p^{r-j}-\alpha^{\ast_{j+1}}p.
\end{equation}
\end{lemma}

\begin{proof}
From the proof of Lemma \ref{u}, we know
$$
d(\alpha^{\ast_r}p^{r-j}-\alpha^{\ast_{j+1}}p)=\<s^{-r}c\>_dp^{r-j}-\<s^{-j-1}c\>_dp\eq s^{-r}cs^{r-j}-s^{-j-1}cs\eq0\pmod{d},
$$
which means the right-hand side of \eqref{dashreduceeq} is an integer divisible by  $p$. Moreover,
$$
-1<\f{s-d+1}{d}\leq\f{p}{d}-\f{d-1}{d}\leq\alpha^{\ast_{j+1}}p-\alpha^{\ast_{j}}<p-\alpha^{\ast_{j}}<p.
$$
Therefore, we have
$$
\alpha^{\ast_{j+1}}p-\alpha^{\ast_{j}}\in\{0,1,2,\ldots,p-1\},
$$
and so
$$
\alpha^{\ast_r}p^{r-j}-\alpha^{\ast_j}-p<\alpha^{\ast_r}p^{r-j}-\alpha^{\ast_{j+1}}p\leq \alpha^{\ast_r}p^{r-j}-\alpha^{\ast_j},
$$
which implies that the right-hand side of \eqref{dashreduceeq} is the largest integer divisible by $p$ in $\{1,2,3,\ldots,$ $ \alpha^{\ast_r}p^{r-j}-\alpha^{\ast_j}\}$.

The proof of Lemma \ref{dashreduce} is now complete.
\end{proof}

Finally, we give two applications of the dash operation in evaluations of the Pochhammer symbols.   
\begin{lemma}\label{pochreduce}
For any odd prime $p$, $x\in\Z_p$ and $r\in\Z^+$, we have
$$
(x)_{p^r}=(-1)^rx^{\ast_r}\prod_{j=1}^r\l(p^{p^{j-1}}\f{\Gamma_p(x+p^j)}{\Gamma_p(x)}\r).
$$
\end{lemma}

\begin{proof}
Recall that for any $x\in\Z_p$, 
$$
\f{\Gamma_p(x+1)}{\Gamma_p(x)}=\begin{cases}-x,\quad&\t{if}\ x\not\eq0\pmod{p},\\ -1,\quad&\t{if}\ x\eq0\pmod{p}\end{cases}
$$
(cf. \cite{Morita1975}). It follows that
\begin{align*}
(x)_{p^r}&=(-1)^{p^r}(x+\<-x\>_p)(x+\<-x\>_p+p)(x+\<-x\>_p+2p)\cdots(x+\<-x\>_p+p^r-p)\f{\Gamma_p(x+p^r)}{\Gamma_p(x)}\\
&=-(x^{\ast})_{p^{r-1}}p^{p^{r-1}}\f{\Gamma_p(x+p^r)}{\Gamma_p(x)}\\
&=(x^{\ast_2})_{p^{r-2}}p^{p^{r-2}}p^{p^{r-1}}\f{\Gamma_p(x+p^{r-1})}{\Gamma_p(x)}\f{\Gamma_p(x+p^r)}{\Gamma_p(x)}\\
&=\cdots\\
&=(-1)^rx^{\ast_r}\prod_{j=1}^r\l(p^{p^{j-1}}\f{\Gamma_p(x+p^j)}{\Gamma_p(x)}\r).
\end{align*}
This concludes the proof.
\end{proof}

\begin{lemma}\label{1/2+alpha}
Under the conditions of Theorem \ref{mainth1}, for $l\in\{0,1,2,\ldots,a-1\}$, modulo $p$ we have
\begin{align*}
&\f{(\f12+\alpha)_l}{(\f12+\alpha+p^r)_l}\\
&\qquad\eq \begin{cases}1,\quad&\t{if}\ a<(p^r+1)/2\ \t{or}\ a\geq(p^r+1)/2\ \t{and}\ l<a-(p^r-1)/2,\\
(2{\alpha}^{\ast_r}-1)/(2{\alpha}^{\ast_r}+1),\quad&\t{if}\ a\geq(p^r+1)/2\ \t{and}\ l\geq a-(p^r-1)/2.\end{cases}
\end{align*}
\end{lemma}

\begin{proof}
It is easy to see that
$$
\f{(\f12+\alpha)_l}{(\f12+\alpha+p^r)_l}=\prod_{j=0}^{l-1}\f{\f12+\alpha+j}{\f12+\alpha+j+p^r}=\prod_{j=0}^{l-1}\f{1}{1+\f{p^r}{\f12+\alpha+j}}.
$$

\noindent{\it Case 1.} $a<(p^r+1)/2$.

\medskip

Now $\<-\f12-\alpha\>_{p^r}=(p^r-1)/2+a$. Thus, for $j\in\{0,1,2,\ldots,l-1\}$,
$$
\ord_p\l(\f12+\alpha+j\r)<r,
$$
and then
$$
\f{(\f12+\alpha)_l}{(\f12+\alpha+p^r)_l}\eq1\pmod{p}.
$$

\noindent{\it Case 2.} $a\geq(p^r+1)/2$ and $l<a-(p^r-1)/2$.

\medskip

Now $\<-\f12-\alpha\>_{p^r}=(p^r-1)/2+a-p^r=a-(p^r+1)/2$. For $j\in\{0,1,2,\ldots,l-1\}$, $j\leq l-1<a-(p^r+1)/2$. Therefore,
$$
\ord_p\l(\f12+\alpha+j\r)<r,
$$
and we also have
$$
\f{(\f12+\alpha)_l}{(\f12+\alpha+p^r)_l}\eq1\pmod{p}.
$$

\noindent{\it Case 3.} $a\geq(p^r+1)/2$ and $l\geq a-(p^r-1)/2$.

\medskip

When $j=a-(p^r+1)/2$, applying Lemma \ref{u}, we obtain $1/2+\alpha+j=p^r(\alpha^{\ast_r}-1/2)$. Meanwhile, if $j\neq a-(p^r+1)/2$,
$$
\ord_p\l(\f12+\alpha+j\r)<r.
$$
Therefore, 
$$
\f{(\f12+\alpha)_l}{(\f12+\alpha+p^r)_l}\eq\f{1}{1+\f{p^r}{p^r(\alpha^{\ast_r}-\f12)}}=\f{2\alpha^{\ast_r}-1}{2\alpha^{\ast_r}+1}\pmod{p}.
$$

The proof of Lemma \ref{1/2+alpha} is now complete.
\end{proof}

\section{Proof of Theorem \ref{mainth1} and Corollary \ref{cor}}\label{sec3}

It is well-known (see, e.g. \cite{W}) that
\begin{equation}\label{2ordharmonic}
H_{p-1}^{(2)}\eq H_{(p-1)/2}^{(2)}\eq 0\pmod{p},
\end{equation}
where $p\geq 5$ is a prime.

The following two lemmas involve reductions of certain harmonic numbers in the sense of congruence.

\begin{lemma}\label{2ordharmonicdown}
For any prime $p\geq 5$ and positive integer $r$, we have
\begin{equation}\label{2ordharmonicdowneq}
p^{2r}H_{p^r-1}^{(2)}\eq p^{2r}H_{(p^r-1)/2}^{(2)}\eq 0\pmod{p^3}.
\end{equation}
\end{lemma}

\begin{proof}
In view of \eqref{2ordharmonic}, \eqref{2ordharmonicdowneq} holds for $r=1$. Now assume $r\geq 2$. By \eqref{2ordharmonic}, we have
\begin{align*}
p^{2r}H_{p^r-1}^{(2)}&=p^{2r}\sum_{\substack{k=1\\ p\nmid k}}^{p^r-1}\f{1}{k^2}+p^{2r}\sum_{\substack{k=1\\ p\mid k}}^{p^r-1}\f{1}{k^2}\\
&\eq  p^{2r-2}\sum_{k=1}^{p^{r-1}-1}\f{1}{k^2}\eq \cdots\eq p^2\sum_{k=1}^{p-1}\f{1}{k^2}\\
&\eq 0\pmod{p^3}
\end{align*}
and
\begin{align*}
p^{2r}H_{(p^r-1)/2}^{(2)}&=p^{2r}\sum_{\substack{k=1\\ p\nmid k}}^{(p^r-1)/2}\f{1}{k^2}+p^{2r}\sum_{\substack{k=1\\ p\mid k}}^{(p^r-1)/2}\f{1}{k^2}\\
&\eq  p^{2r-2}\sum_{k=1}^{(p^{r-1}-1)/2}\f{1}{k^2}\eq \cdots\eq p^2\sum_{k=1}^{(p-1)/2}\f{1}{k^2}\\
&\eq 0\pmod{p^3}.
\end{align*}
This concludes the proof.
\end{proof}

\begin{lemma}\label{harmonicreduce}
Under the conditions of Theorem \ref{mainth1}, we have
\begin{align}
\label{harmonicreduce1eq}p^{2r}\sum_{l=1}^{a}\f{1}{(\alpha+a-l)^2}&\eq p^2H_{\alpha^{\ast_r}p-\alpha^{\ast_{r-1}}}^{(2)}\pmod{p^3},\\
\label{harmonicreduce2eq}p^{2r}\sum_{l=1}^{(p^r-1)/2}\f{1}{(\alpha+a-l)^2}&\eq 0\pmod{p^3}.
\end{align}
\end{lemma}

\begin{proof}
We first prove \eqref{harmonicreduce1eq}. Clearly, if $a=0$, \eqref{harmonicreduce1eq} holds trivially. If $r=1$, then $a=\<-\alpha\>_p=\alpha^{\ast}p-\alpha$. And we easily obtain
$$
p^2\sum_{l=1}^{\alpha^{\ast}p-\alpha}\f{1}{(\alpha+\<-\alpha\>_p-l)^2}\eq p^2\sum_{l=1}^{\alpha^{\ast}p-\alpha}\f{1}{(-l)^2}=p^2H_{\alpha^{\ast}p-\alpha}^{(2)}\pmod{p^3}.
$$

Assume $r>1$ and $a>0$. Now, $a=\alpha^{\ast_r}p^r-\alpha$, and then by Lemma \ref{dashreduce},
\begin{align*}
p^{2r}\sum_{l=1}^{\alpha^{\ast_r}p^r-\alpha}\f{1}{(\alpha^{\ast_r}p^r-l)^2}&=p^{2r}\sum_{\substack{l=1\\ p\nmid l}}^{\alpha^{\ast_r}p^r-\alpha}\f{1}{(\alpha^{\ast_r}p^r-l)^2}+p^{2r}\sum_{\substack{l=1\\ p\mid l}}^{\alpha^{\ast_r}p^r-\alpha}\f{1}{(\alpha^{\ast_r}p^r-l)^2}\\
&\eq p^{2r}\sum_{\substack{l=1\\ p\mid l}}^{\alpha^{\ast_r}p^r-\alpha}\f{1}{(\alpha^{\ast_r}p^r-l)^2}\\
&=p^{2r-2}\sum_{l=1}^{\alpha^{\ast_r}p^{r-1}-\alpha^{\ast}}\f{1}{(\alpha^{\ast_r}p^{r-1}-l)^2}\\
&=p^{2r-2}\sum_{\substack{l=1\\ p\nmid l}}^{\alpha^{\ast_r}p^{r-1}-\alpha^{\ast}}\f{1}{(\alpha^{\ast_r}p^{r-1}-l)^2}+p^{2r-2}\sum_{\substack{l=1\\ p\mid l}}^{\alpha^{\ast_r}p^{r-1}-\alpha^{\ast}}\f{1}{(\alpha^{\ast_r}p^{r-1}-l)^2}\\
&\eq p^{2r-2}\sum_{\substack{l=1\\ p\mid l}}^{\alpha^{\ast_r}p^{r-1}-\alpha^{\ast}}\f{1}{(\alpha^{\ast_r}p^{r-1}-l)^2}\\
&=p^{2r-4}\sum_{l=1}^{\alpha^{\ast_r}p^{r-2}-\alpha^{\ast_2}}\f{1}{(\alpha^{\ast_r}p^{r-2}-l)^2}\\
&\eq\cdots\\
&\eq p^2\sum_{l=1}^{\alpha^{\ast_r}p-\alpha^{\ast_{r-1}}}\f{1}{(\alpha^{\ast_r}p-l)^2}\\
&\eq p^2H_{\alpha^{\ast_r}p-\alpha^{\ast_{r-1}}}^{(2)}\pmod{p^3},
\end{align*}
as desired.

Now we prove \eqref{harmonicreduce2eq}. Obviously, by \eqref{2ordharmonic}, \eqref{harmonicreduce2eq} holds for $r=1$. Suppose $r>1$. Then, by \eqref{2ordharmonic}, we arrive at
\begin{align*}
p^{2r}\sum_{l=1}^{(p^r-1)/2}\f{1}{(\alpha+a-l)^2}&=p^{2r}\sum_{\substack{l=1\\p\nmid l}}^{(p^r-1)/2}\f{1}{(\alpha^{\ast_r}p^r-l)^2}+p^{2r}\sum_{\substack{l=1\\p\mid l}}^{(p^r-1)/2}\f{1}{(\alpha^{\ast_r}p^r-l)^2}\\
&\eq p^{2r}\sum_{\substack{l=1\\p\mid l}}^{(p^r-1)/2}\f{1}{(\alpha^{\ast_r}p^r-l)^2}\\
&= p^{2r-2}\sum_{l=1}^{(p^{r-1}-1)/2}\f{1}{(\alpha^{\ast_r}p^{r-1}-l)^2}\\
&\eq \cdots\\
&\eq p^2\sum_{l=1}^{(p-1)/2}\f{1}{(\alpha^{\ast_r}p-l)^2}\\
&\eq 0\pmod{p^3},
\end{align*}
as desired.

The proof of Lemma \ref{harmonicreduce} is now complete.
\end{proof}

For $k\in\N$ and $x\not\in-\f12-\N=\{-\f12,-\f32,-\f52,\ldots\}$, define
$$
F(x,k)=(2k+x)\f{(x)_k^3(\f12)_k}{(1)_k^3(\f12+x)_k}
$$
and
$$
G(x,k)=\f{k^3(k+2x)}{x^3}\f{(x)_k^3(\f12)_k}{(1)_k^3(\f12+x)_k}.
$$
It is easy to verify that
\begin{equation}\label{WZpair}
F(x+1,k)-F(x,k)=G(x,k+1)-G(x,k),
\end{equation}
that is, $(F,G)$ forms a WZ pair.

The last two lemmas are devoted to establish supercongruences of sums of $F$ and $G$.
\begin{lemma}\label{evalF}
Under the conditions of Theorem \ref{mainth1}, we have
$$
\sum_{k=0}^{p^r-1}F(\alpha^{\ast_r}p^r,k)\eq \alpha^{\ast_r}p^r\pmod{p^{r+3}}.
$$
\end{lemma}

\begin{proof}
Obviously,
\begin{align*}
\sum_{k=0}^{p^r-1}F(\alpha^{\ast_r}p^r,k)&=\sum_{k=0}^{p^r-1}(2k+\alpha^{\ast_r}p^r)\f{(\alpha^{\ast_r}p^r)_k^3(\f12)_k}{(1)_k^3(\f12+\alpha^{\ast_r}p^r)_k}\\
&=\alpha^{\ast_r}p^r+p^{3r}{\alpha^{\ast_r}}^3\sum_{k=1}^{p^r-1}\f{2k+\alpha^{\ast_r}p^r}{k^3}\f{(1+\alpha^{\ast_r}p^r)_{k-1}^3(\f12)_k}{(1)_{k-1}^3(\f12+\alpha^{\ast_r}p^r)_k}.
\end{align*}
Note that for $k\in\{1,2,\ldots,p^r-1\}$, $\ord_p(k)\leq r-1$, where $\ord_p$ stands for the $p$-adic order. Therefore,
$$
\ord_p\l(\f{p^{3r}}{k^2}\r)\geq r+2\quad \t{and}\quad \ord_p\l(\f{p^{4r}}{k^3}\r)\geq r+3.
$$
Moreover, it is easy to see that
$$
\f{(1+\alpha^{\ast_r}p^r)_{k-1}}{(1)_{k-1}^3}=\prod_{j=1}^{k-1}\f{j+\alpha^{\ast_r}p^r}{j}=\prod_{j=1}^{k-1}\l(1+\f{\alpha^{\ast_r}p^r}{j}\r)\eq1\pmod{p}
$$
and
\begin{align*}
\f{(\f12)_k}{(\f12+\alpha^{\ast_r}p^r)_k}&=\prod_{j=0}^{k-1}\f{\f12+j}{\f12+j+\alpha^{\ast_r}p^r}=\prod_{j=0}^{k-1}\f{1}{1+\f{\alpha^{\ast_r}p^r}{\f12+j}}\\
&\eq\begin{cases}1/(2\alpha^{\ast_r}+1)\pmod{p}\quad&\t{if}\ k\geq (p^r+1)/2,\\ 1\pmod{p}\quad&\t{if}\ k\leq (p^r-1)/2.\end{cases}
\end{align*}
Combining the above and in view of Lemmas \ref{u} and  \ref{2ordharmonicdown}, we arrive at
$$
\sum_{k=0}^{p^r-1}F(\alpha^{\ast_r}p^r,k)\eq \alpha^{\ast_r}p^r+2(\alpha^{\ast_r})^3p^{3r}\sum_{k=1}^{(p^r-1)/2}\f{1}{k^2}+\f{2(\alpha^{\ast_r})^3p^{3r}}{2\alpha^{\ast_r}+1}\sum_{k=(p^r+1)/2}^{p^r-1}\f{1}{k^2}\eq  \alpha^{\ast_r}p^r \pmod{p^{r+3}},
$$
as desired.
\end{proof}

\begin{lemma}\label{evalG}
Under the conditions of Theorem \ref{mainth1}, we have
\begin{equation}\label{evalGeq}
\sum_{l=0}^{a-1}G(\alpha+l,p^r)\eq \f{(\alpha^{\ast_r})^3}{(\f12+\alpha)^{\ast_r}}p^{r+2}H_{\alpha^{\ast_r}p-\alpha^{\ast_{r-1}}}^{(2)}\pmod{p^{r+3}}.
\end{equation}
\end{lemma}

\begin{proof}
Clearly,
\begin{align*}
\sum_{l=0}^{a-1}G(\alpha+l,p^r)&=\sum_{l=0}^{a-1}\f{p^{3r}(p^r+2\alpha+2l)}{(\alpha+l)^3}\f{(\alpha+l)_{p^r}^3(\f12)_{p^r}}{(1)_{p^r}^3(\f12+\alpha+l)_{p^r}}\\
&=\f{p^{3r}(\alpha)_{p^r}^3(\f12)_{p^r}}{(1)_{p^r}^3(\f12+\alpha)_{p^r}}\sum_{l=0}^{a-1}\f{(p^r+2\alpha+2l)(\alpha+p^r)_l^3(\f12+\alpha)_l}{(\alpha+l)^3(\alpha)_l^3(\f12+\alpha+p^r)_l}.
\end{align*}
Then, by Lemma \ref{pochreduce}, we have
\begin{align*}
\sum_{l=0}^{a-1}G(\alpha+l,p^r)&=\f{p^{3r}(\alpha^{\ast_r})^3}{2(\f12+\alpha)^{\ast_r}}\prod_{j=1}^r\f{\Gamma_p(\alpha+p^j)^3\Gamma_p(\f12+p^j)\Gamma_p(1)^3\Gamma_p(\f12+\alpha)}{\Gamma_p(\alpha)^3\Gamma_p(\f12)\Gamma_p(1+p^j)^3\Gamma_p(\f12+\alpha+p^j)}\\
&\quad\times \sum_{l=0}^{a-1}\f{(p^r+2\alpha+2l)(\alpha+p^r)_l^3(\f12+\alpha)_l}{(\alpha+l)^3(\alpha)_l^3(\f12+\alpha+p^r)_l},
\end{align*}
where we have used the facts
$$
\l(\f12\r)^{\ast}=\f12\quad \t{and}\quad 1^{\ast}=1.
$$
Since $(\f12+\alpha)^{\ast_r}\not\eq0\pmod{p}$, we have
$$
\ord_p\l(\f{p^{3r}(\alpha^{\ast_r})^3}{(\f12+\alpha)^{\ast_r}}\r)\geq 3r.
$$
It is easy to see that $\ord_p(\alpha+l)\leq r-1$ for $l\in\{0,1,2,\ldots,a-1\}$. It follows that
$$
\ord_p\l(\f{p^{4r}}{(\alpha+l)^3}\r)\geq r+3
$$
and
$$
\ord_p\l(\f{p^{3r}}{(\alpha+l)^2}\r)\geq r+2.
$$
For $k\in\{0,1,2,\ldots,a-1\}$ we have
$$
\f{(\alpha+p^r)_l}{(\alpha)_l}=\prod_{j=0}^{l-1}\f{\alpha+j+p^r}{\alpha+j}=\prod_{j=0}^{l-1}\l(1+\f{p^r}{\alpha+j}\r)\eq1\pmod{p}.
$$
In view of Lemma \ref{1/2+alpha} and the above, we obtain
\begin{equation}\label{evalGkey}
\sum_{l=0}^{a-1}G(\alpha+l,p^r)\eq\f{p^{3r}(\alpha^{\ast_r})^3}{(\f12+\alpha)^{\ast_r}}\sum_{l=0}^{a-1}\f{(\f12+\alpha)_l}{(\alpha+l)^2(\f12+\alpha+p^r)_l}\pmod{p^{r+3}}.
\end{equation}

Below we divide the proof into two cases.

\medskip

\noindent{\it Case 1}. $a<(p^r+1)/2$.

\medskip

By Lemma \ref{1/2+alpha},
$$
\f{(\f12+\alpha)_l}{(\f12+\alpha+p^r)_l}\eq1\pmod{p}.
$$
This, together with \eqref{evalGkey}, gives
\begin{align*}
\sum_{l=0}^{a-1}G(\alpha+l,p^r)&\eq \f{p^{3r}(\alpha^{\ast_r})^3}{(\f12+\alpha)^{\ast_r}}\sum_{l=0}^{a-1}\f{1}{(\alpha+l)^2}\\
&=\f{p^{3r}(\alpha^{\ast_r})^3}{(\f12+\alpha)^{\ast_r}}\sum_{l=1}^{a}\f{1}{(\alpha+a-l)^2}\pmod{p^{r+3}}.
\end{align*}
Then, by \eqref{harmonicreduce1eq}, we obtain \eqref{evalGeq} in this case.

\medskip

\noindent{\it Case 2}. $a\geq(p^r+1)/2$. 

\medskip

By Lemma \ref{1/2+alpha},
$$
\f{(\f12+\alpha)_l}{(\f12+\alpha+p^r)_l}\eq\begin{cases}1\pmod{p},\quad&\t{if}\ l<a-(p^r-1)/2,\\ (2\alpha^{\ast_r}-1)/(2\alpha^{\ast_r}+1)\pmod{p},\quad&\t{if}\ l\geq a-(p^r-1)/2.\end{cases}
$$
Substituting this into \eqref{evalGkey} and using \eqref{harmonicreduce2eq} we have
\begin{align*}
\sum_{l=0}^{a-1}G(\alpha+l,p^r)&\eq \f{p^{3r}(\alpha^{\ast_r})^3}{(\f12+\alpha)^{\ast_r}}\sum_{l=0}^{a-(p^r+1)/2}\f{1}{(\alpha+l)^2}+\f{p^{3r}(\alpha^{\ast_r})^3(2\alpha^{\ast_r}-1)}{(\f12+\alpha)^{\ast_r}(2\alpha^{\ast_r}+1)}\sum_{l=a-(p^r-1)/2}^{a-1}\f{1}{(\alpha+l)^2}\\
&=\f{p^{3r}(\alpha^{\ast_r})^3}{(\f12+\alpha)^{\ast_r}}\sum_{l=(p^r+1)/2}^{a}\f{1}{(\alpha+a-l)^2}+\f{p^{3r}(\alpha^{\ast_r})^3(2\alpha^{\ast_r}-1)}{(\f12+\alpha)^{\ast_r}(2\alpha^{\ast_r}+1)}\sum_{l=1}^{(p^r-1)/2}\f{1}{(\alpha+a-l)^2}\\
&\eq \f{p^{3r}(\alpha^{\ast_r})^3}{(\f12+\alpha)^{\ast_r}}\sum_{l=1}^{a}\f{1}{(\alpha+a-l)^2}\pmod{p^{r+3}}.
\end{align*}
With the help of \eqref{harmonicreduce1eq}, we obtain the desired result.

The proof of Lemma \ref{evalG} is now complete.
\end{proof}

\medskip

\noindent{\it Proof of Theorem \ref{mainth1}}. By Lemma \ref{u}, $\alpha+a=\alpha^{\ast_r}p^r$. In view of \eqref{WZpair}, we have
\begin{align*}
\sum_{k=0}^{p^r-1}F(\alpha,k)&=\sum_{k=0}^{p^r-1}(F(\alpha,k)-F(\alpha^{\ast_r}p^r,k))+\sum_{k=0}^{p^r-1}F(\alpha^{\ast_r}p^r,k)\\
&=\sum_{k=0}^{p^r-1}\sum_{l=0}^{a-1}(F(\alpha+l,k)-F(\alpha+l+1,k))+\sum_{k=0}^{p^r-1}F(\alpha^{\ast_r}p^r,k)\\
&=\sum_{l=0}^{a-1}\sum_{k=0}^{p^r-1}(G(\alpha+l,k)-G(\alpha+l,k+1))+\sum_{k=0}^{p^r-1}F(\alpha^{\ast_r}p^r,k)\\
&=\sum_{l=0}^{a-1}(G(\alpha+l,0)-G(\alpha+l,p^r))+\sum_{k=0}^{p^r-1}F(\alpha^{\ast_r}p^r,k)\\
&=\sum_{k=0}^{p^r-1}F(\alpha^{\ast_r}p^r,k)-\sum_{l=0}^{a-1}G(\alpha+l,p^r),
\end{align*}
where we have used the fact $G(x,0)=0$. By Lemmas \ref{evalF} and \ref{evalG}, we immediately obtain
$$
\sum_{k=0}^{p^r-1}F(\alpha,k)=\alpha^{\ast_r}p^r -\f{(\alpha^{\ast_r})^3}{(\f12+\alpha)^{\ast_r}}p^{r+2}H_{\alpha^{\ast_r}p-\alpha^{\ast_{r-1}}}^{(2)}\pmod{p^{r+3}},
$$
and conclude the proof.\qed

\medskip

\noindent{\it Proof of Corollary \ref{cor}}. Putting $d=4,s=3,c=1$ and requiring $r$ to be odd in Theorem \ref{mainth1}, we obtain
\begin{equation}\label{cor2}
\sum_{k=0}^{p^r-1}(8k+1)\f{(\f14)_k^3(\f12)_k}{(1)_k^3(\f34)_k}\eq 3p^r-\f{27}{4}p^{r+2}H_{(3p-1)/4}^{(2)} \pmod{p^{r+3}}.
\end{equation}
Via a similar argument as in the proof of Lemma \ref{harmonicreduce}, one has 
$$
p^{2r}\sum_{j=1}^{(p^r-3)/4}\f{1}{j^2}\eq p^2\sum_{j=1}^{(p-3)/4}\f{1}{j^2}\pmod{p^3}.
$$
Moreover, by \eqref{2ordharmonic},
$$
\sum_{j=1}^{(p-3)/4}\f{1}{j^2}=H_{(p-1)/2}^{(2)}-\sum_{j=(p+1)/4}^{p-1}\f{1}{j^2}\eq-\sum_{j=1}^{(3p-1)/4}\f{1}{(p-j)^2}\eq -H_{(3p-1)/4}^{(2)} \pmod{p}.
$$
Therefore,
$$
\f{27}{4}p^{r+2}H_{(3p-1)/4}^{(2)} \eq -\f{27}{4}p^{3r}\sum_{j=1}^{(p^r-3)/4}\f{1}{j^2}\pmod{p^{r+3}}.
$$
Substituting this into \eqref{cor2} we immediately obtain \eqref{GuoZhaoconj7.2} in the case $p\geq5$.

Suppose that $p=3$. By a similar argument as in the proof of Lemma \ref{evalF}, we have
\begin{align*}
\sum_{k=0}^{p^r-1}F(\alpha^{\ast_r}p^r,k)&\eq \alpha^{\ast_r}p^r+2(\alpha^{\ast_r})^3p^{3r}\sum_{k=1}^{(p^r-1)/2}\f{1}{k^2}+\f{2(\alpha^{\ast_r})^3p^{3r}}{2\alpha^{\ast_r}+1}\sum_{k=(p^r+1)/2}^{p^r-1}\f{1}{k^2}\pmod{p^{r+3}}.
\end{align*}
Note that 
$$
\ord_p\l(\sum_{k=1}^{(p^r-1)/2}\f{1}{k^2}\r)\geq -2(r-1)\quad\t{and}\quad \ord_p\l(\sum_{k=(p^r+1)/2}^{p^r-1}\f{1}{k^2}\r)\geq -2(r-1).
$$
Since $\alpha^{\ast_r}=3/4\eq0\pmod{p}$ and $2\alpha^{\ast_r}+1=5/2\not\eq0\pmod{p}$, we still have
\begin{equation}\label{evalF'}
\sum_{k=0}^{p^r-1}F(\alpha^{\ast_r}p^r,k)\eq \alpha^{\ast_r}p^r\pmod{p^{r+3}}.
\end{equation}
Now, $a=(3p^r-1)/4\geq (p^r+1)/2$. From the proof of Lemma \ref{evalG}, we know
\begin{align*}
&\sum_{l=0}^{a-1}G(\alpha+l,p^r)\\
&\qquad\eq \f{p^{3r}(\alpha^{\ast_r})^3}{(\f12+\alpha)^{\ast_r}}\sum_{l=(p^r+1)/2}^{a}\f{1}{(\alpha+a-l)^2}+\f{p^{3r}(\alpha^{\ast_r})^3(2\alpha^{\ast_r}-1)}{(\f12+\alpha)^{\ast_r}(2\alpha^{\ast_r}+1)}\sum_{l=1}^{(p^r-1)/2}\f{1}{(\alpha+a-l)^2}\pmod{p^{r+3}}.
\end{align*}
Also, 
$$
\ord_p\l(\sum_{l=(p^r+1)/2}^{a}\f{1}{(\alpha+a-l)^2}\r)\geq -2(r-1)\quad\t{and}\quad \ord_p\l(\sum_{l=1}^{(p^r-1)/2}\f{1}{(\alpha+a-l)^2}\r)\geq -2(r-1).
$$
Then, in view of the facts $\alpha^{\ast_r}\eq0\pmod{p}$, $(1/2+\alpha)^{\ast_r}=1/4\not\eq0\pmod{p}$ and $2\alpha^{\ast_r}+1\not\eq0\pmod{p}$, we obtain
$$
\sum_{l=0}^{a-1}G(\alpha+l,p^r)\eq0\pmod{p^{r+3}}.
$$
This, together with \eqref{evalF'} gives
$$
\sum_{k=0}^{p^r-1}F(\alpha,k)=\sum_{k=0}^{p^r-1}F(\alpha^{\ast_r}p^r,k)-\sum_{l=0}^{a-1}G(\alpha+l,p^r)\eq\alpha^{\ast_r}p^r\pmod{p^{r+3}},
$$
which implies
\begin{equation*}
\sum_{k=0}^{3^r-1}(8k+1)\f{(\f14)_k^3(\f12)_k}{(1)_k^3(\f34)_k}\eq 3^{r+1}\pmod{3^{r+3}}.
\end{equation*}
Moreover, since 
$$
\ord_3\l(\sum_{j=1}^{(3^r-3)/4}\f{1}{j^2}\r)\geq-2(r-1),
$$
we have
$$
\f{27}{4}3^{3r}\sum_{j=1}^{(3^r-3)/4}\f{1}{j^2}\eq 0\pmod{3^{r+3}}.
$$
Therefore, we arrive at
\begin{equation*}
\sum_{k=0}^{3^r-1}(8k+1)\f{(\f14)_k^3(\f12)_k}{(1)_k^3(\f34)_k}\eq 3^{r+1}+\f{27}{4}3^{3r}\sum_{j=1}^{(3^r-3)/4}\f{1}{j^2}\pmod{3^{r+3}},
\end{equation*}
as desired.

The proof of Corollary \ref{cor} is now complete.\qed

\begin{Acks}
This work is supported by the National Natural Science Foundation of China (grant 12201301).
\end{Acks}

\end{document}